\documentclass{amsart}

\usepackage{
amssymb,
amsfonts,
amsmath,
amsthm,
mathtools,
commath
}
\usepackage[pagewise]{lineno}
\usepackage[utf8]{inputenc}	
\usepackage{hyperref, color}
\hypersetup{
	colorlinks = true,
	linkcolor = blue, %citações de equações e teoremas
	filecolor = red,
	urlcolor = red,
	citecolor = blue %citações de referências
}

\newcommand{\R}{\mathbb R}

\newcommand{\TT}{\mathrm T}
\newcommand{\Q}{\mathbb Q}

\newcommand{\N}{\mathbb N}
\newcommand{\h}{\mathbb H}
\newcommand{\s}{\mathbb S}

\newcommand{\id}{\mathrm{Id}}
\newcommand{\Ric}{\mathrm{Ric}}
\newcommand{\Scal}{\mathrm{Scal}}
\newcommand{\trace}{\mathrm{tr} \,}
\newcommand{\dotprod}[2]{\langle #1, #2 \rangle}

\newtheorem{theorem}{Theorem}[section]
\newtheorem{lemma}[theorem]{Lemma}
\newtheorem{proposition}[theorem]{Proposition}
\newtheorem{corollary}[theorem]{Corollary}

\theoremstyle{definition}

\newtheorem{remark}[theorem]{Remark}

\numberwithin{equation}{section}
\numberwithin{figure}{section}

\begin{document}
%\linenumbers
%Remarks on
\title[Minimal Kähler submanifolds in product of Space Forms]{Minimal Kähler submanifolds in product of\\ Space Forms}

\author[Alcides de Carvalho]{Alcides de Carvalho}
\address{Universidade de S\~ao Paulo \\
ICMC-USP \\
Av. Trabalhador S\~ao-Carlense 400 \\
13560-970 S\~ao Carlos – SP \\
 Brazil}
\email{alcidesj@impa.br}

\author[Iury Domingos]{Iury Domingos}
\address{Universidade Federal da Paraíba\\
   Departamento de Matemática\\
   58051-900 João Pessoa - Paraíba\\
   Brazil}
\email{domingos1@univ-lorraine.fr}

\thanks{The first author is supported by Fapesp grant 2019/177-0.}



% 53C42  	Differential geometry of immersions
% 53B25  	Local submanifolds
% 53C40  	Global submanifolds
% 53B35  	Local differential geometry of Hermitian and Kählerian structures
% 32Q15  	Kähler manifolds
% 53C55  	Global differential geometry of Hermitian and Kählerian manifolds

\begin{abstract} In this article, we study minimal isometric immersions of Kähler manifolds into product of two real space forms. We analyse the obstruction conditions to the existence of pluriharmonic isometric immersions of a Kähler manifold into those spaces and we prove that the only ones into $\s^{m-1}\times \R$ and $\h^{m-1}\times \R$ are the minimal isometric immersions of Riemannian surfaces. Futhermore, we show that the existence of a minimal isometric immersion of a Kähler manifold $M^{2n}$ into $\mathbb{S}^{m-1}\times\R$ and $\mathbb{S}^{m-k}\times \mathbb{H}^k$ imposes strong restrictions on the Ricci and scalar curvatures of $M^{2n}$. In this direction, we characterise some cases as either isometric immersions with parallel second fundamental form or anti-pluriharmonic isometric immersions.
\end{abstract}

\maketitle

\section{Introduction}\label{introduction}

This paper deals with minimal isometric immersions of a Kähler manifold into product of two real space forms. More specifically, we will be interested first with obstructions to the existence of pluriharmonic isometric immersions and secondly with restrictions on the Ricci curvature and scalar curvature of minimal Kähler isometric immersions into those spaces.

Over the years, pluriharmonic isometric immersions have been studied by several authors. In the literature, they are also called by $(1,1)$-geodesic immersions and circular immersions  (cf. \cite{Dajczer-Gromoll-85,FerreiraRigoliTribuzy93}). Those immersions appear as a natural extension of minimal immersions of Riemann surfaces into a target space, therefore they are minimal in the classical sense. The simplest examples of pluriharmonic isometric immersions are orientable minimal surfaces in an arbitrary Riemannian manifolds and holomorphic isometric immersions between Kähler manifolds, and it is important to notice that those immersions also have associated families of pluriharmonic isometric immersions when the ambient manifold is a Riemannian symmetric space (cf. \cite{Esch-Tribuzy98}).

In real space forms, the study of pluriharmonic isometric immersions and their associated families is due to Dajczer and Gromoll (cf.  \cite{Dajczer-Gromoll-85}). They proved that for each non-holomorphic pluriharmonic isometric immersion into real space forms there exists a one-parameter family of noncongruent pluriharmonic submanifolds, likewise minimal surfaces in three-dimensional space forms. %Moreover, the set of all pluriharmonic isometric immersions of a given Kähler manifold into Euclidean space have a unique holomorphic representative.
Dajczer and Rodríguez proved another interesting fact about pluriharmonic isometric immersions of Kähler manifolds into Euclidean spaces. They showed that pluriharmonic and minimal Kähler submanifolds mean the same in Euclidean spaces, although this is not obvious (cf. \cite{DajczerRodriguez86}). In addition, they proved that the only minimal isometric immersions of Kähler
manifold $M^{2n}$ into $\h^m$ are the minimal isometric immersions
of Riemannian surfaces.

More generally for locally symmetric Riemannian manifold of non-compact type, Ferreira, Rigoli and Tribuzy showed that pluriharmonic isometric immersions and minimal isometric immersions of Kähler manifold $M^{2n}$ into those spaces are also the same objects (cf. \cite{FerreiraRigoliTribuzy93}). Under some assumptions on the Ricci and scalar curvature of those target spaces, they proved additionally that the only pluriharmonic isometric immersion of Kähler manifold $M^{2n}$ into conformally flat Riemannian manifolds are the minimal isometric immersions of Riemannian surfaces.

In a seminal work, Takahashi established a necessary condition on the Ricci curvature $\Ric_M$ for that a given Riemannian manifold $M^n$ admits a minimal isometric immersion into a real space form of constant sectional curvature $c$ (cf. \cite{Takahashi}). This geometric restriction appears naturally as a consequence of Gauss equation for minimal isometric immersions and, up to normalization, the Ricci curvature must satisfy $\Ric_M\leq c(n-1)$, with $n\geq 2$. In another direction of \cite{DajczerRodriguez86}, Dajczer and Rodríguez proved if we want to immerse minimally in $\s^m$ a Kähler manifold $M^{2n}$, this necessary condition will be more restrictive. In this case, the Ricci curvature must satisfy $\Ric_M\leq nc$, with $n\geq 1$. In both works, under assumption of the existence of a minimal isometric immersion, the authors characterise the equality case as a totally geodesic isometric immersion (Takahashi theorem) and as an isometric immersion with parallel second fundamental form (Dajczer-Rodríguez theorem).

The aim of the work is to generalize these results to some products of real space forms.

First, we show that the only pluriharmonic isometric immersions of a Kähler manifold $M^{2n}$ into $\s^{m-1}\times \R$ and $\h^{m-1}\times \R$ are the minimal isometric immersions of Riemannian surfaces. We remark that minimal and pluriharmonic isometric immersions of a Kähler manifold into $\h^{m-1}\times \R$ are the same objects, by Ferreira-Rigoli-Tribuzy results. Dual results are obtained for maps into $\s^{m-k}\times\h^{k}$ and into warped product manifolds $I\times_{\rho}\mathbb{R}^{m-1}$, $I\times_{\rho}\mathbb{S}^{m-1}$ and $I\times_{\rho}\mathbb{H}^{m-1}$, where $I\subset\R$ is an interval, under some additional hypotheses.

Furthermore, we discuss how the existence of a minimal isometric immersion of a Kähler manifold $M^{2n}$ into $\mathbb{S}^{m-1}\times\R$ and $\mathbb{S}^{m-k}\times \mathbb{H}^k$ can to impose strong restrictions on the Ricci curvature and the scalar curvature of $M^{2n}$. In this direction, thanks to the complex structure of $M^{2n}$, we obtain a better upper bound of Ricci curvature of minimal isometric immersions of Kähler manifolds into those manifolds, and we characterise the equality case as isometric immersions with parallel second fundamental form. We also obtain an improvement to the upper bound of scalar curvature, and we characterise the equality case as anti-pluriharmonic isometric immersions. Moreover, we observe that our technique generalize those results to isometric immersions of a Kähler manifold into conformally flat Riemannian manifolds. This case was studied by Ferreira, Rigoli and Tribuzy with certain bounds assumptions on the Ricci curvature of the ambient space.

\section{Preliminaries}\label{preliminaries}

Let $c_1,c_2\in\R$ and $n_1,n_2\in\N$. We denote by $\mathbb{Q}^{n_i}_{c_i}$ be the $n_i$-dimensional simply connected Riemannian manifold of constant sectional curvature $c_i$, for $i=1,2$. As usual,  $\Q^{n}_{c}=\mathbb{S}^{n}_{c}$ is the $n$-Sphere for $c>0$, $\Q^{n}_{c}=\mathbb{R}^{n}$ is the Euclidean $n$-space for $c=0$ and $\Q^{n}_{c}=\mathbb{H}^{n}_{c}$ is the Hyperbolic $n$-space for $c<0$. Finally, we consider $\mathbb{Q}^m = \mathbb{Q}^{n_1}_{c_1}\times \mathbb{Q}^{n_2}_{c_2}$ be the Riemannian product manifold endowed with the product metric, denoted by $\dotprod{\cdot}{\cdot}$, where $m=n_1+n_2$, and let $\pi_i:\Q^m\to\Q^{n_i}_{c_i}$ be the projection onto the factor $\mathbb{Q}^{n_i}_{c_i}$, for $i=1,2$.

Throughout this work, we consider $(M^{2n},\dif s^2)$ a $2n$-dimensional simply connected \emph{Kähler} manifold with almost complex structure $J$, with $n\in\N$. This means that $M$ is a $2n$-dimensional simply connected smooth manifold, endowed with a Riemannian metric $\dif s^2$ (also denoted by $\dotprod{\cdot}{\cdot}$), such that the almost complex structure $J$ is a parallel orthogonal tensor on the tangent bundle of $M$, i.e.,
$J^2 = -\id_{\TT M}$,
\[\langle JX,JY\rangle  = 	\langle X, Y \rangle\]
and
\[(\nabla_X J)Y = \nabla_X JY - J\nabla_X Y = 0,\]
for all $X, Y \in \mathfrak{X}(M)$, where $\id_{\TT M}$ is the identity tensor on $\mathrm{T}M$, $\nabla$ denotes the Riemannian connection of $M$ and $\mathfrak{X}(M)=\Gamma(\TT M)$ denotes the section of $\TT M$.

We fix the Riemann curvature tensor $\mathcal{R}$ of $M$, given by
\[\mathcal{R}(X,Y)Z = \nabla_X \nabla_Y Z- \nabla_Y \nabla_X Z -\nabla_{[X,Y]}Z,\]
and the Ricci tensor $\Ric$ of $M$, is given by
\[\Ric(X,Y) = \text{ trace of the mapping } Z\mapsto \mathcal{R}(Z,X)Y,\]
for all $X,Y,Z\in\mathfrak{X}(M)$. In our convention, we consider the Ricci curvature, in the direction of a unit vector $X\in\mathfrak{X}(M)$, by the contraction of the Ricci tensor, i.e.,
\[\Ric(X)=\Ric(X,X),\]
and the scalar curvature function of $M$ by the trace of the Ricci curvature, i.e.,
\[\Scal = \sum_{i=1}^{2n} \Ric(X_i),\]
where $\{X_1,\ldots,X_{2n}\}$ is an orthonormal frame of $\mathrm{T} M$. In particular, when $n=1$, we have that
\[\Ric (X_i) = K_M \ \text{and} \ \Scal = 2K_M,\]
for $i=1,2$, where $K_M$ denotes the intrinsic curvature of $\dif s^2$.

We remark that the almost complex structure $J$ and the Riemann curvature tensor $\mathcal{R}$ satisfy
\[\mathcal{R}(X,Y)\circ J = J\circ\mathcal{R}(X,Y) \text{ \ and  \ } \mathcal{R}(JX,JY)=\mathcal{R}(X,Y),\]
and the almost complex structure $J$ and the Ricci tensor $\Ric$ satisfy
\[\Ric(JX,JY)=\Ric(X,Y),\]
for all $X, Y \in \mathfrak{X}(M)$ (for more details, the reader may refer to \cite[Chapter 9]{MR1393941}).

Given an isometric immersion $f: M^{2n} \to \mathbb{Q}^m$, $2n<m$, we denote by $\mathcal{R}^{\perp}$ the curvature tensor of the normal bundle  $\TT^{\perp}M$, by $\alpha$, seen as section of the bundle $\TT^{\ast}M\oplus \TT^{\ast}M\oplus \TT^{\perp}M,$ the second fundamental form of $f$
and by $A_{\xi}$ its Weingarten operator in the normal direction $\xi\in\mathfrak{X}(M)^\perp$, given by
$$\dotprod{A_\xi X}{Y} = \dotprod{\alpha(X, Y)}{\xi},$$
for all $X, Y \in \mathfrak{X}(M)$, where $\mathfrak{X}(M)^\perp=\Gamma(\TT^\perp M)$ denotes the section of $\TT^\perp M$.

In order to obtain a Bonnet-type theorem for isometric immersions into $\mathbb{Q}^m$ (and, more generally, considering the signature cases), Lira, Tojeiro and Vitório introduced in \cite{LiraTojeiroVitorio10} the tensors $R$, $S$ and $T$ defined by
\[R = L^t L, \ S = K^t L \text{ \ and \ } T = K^tK,\]
where
\[L = \dif\pi_2\circ f_{\ast} \in \Gamma(\TT^{\ast}M \oplus \TT\mathbb{Q}^{n_2}_{c_2,\mu_2}) \text{ \ and \ } K = \dif\pi_2|_{\TT^{\perp}M}\in \Gamma((\TT^{\perp}M)^{\ast}\oplus \TT\mathbb{Q}^{n_2}_{c_2,\mu_2} ).\]
The tensors $R$ and
$T$ are non-negative symmetric operators whose eigenvalues lie in $[0,1]$. In particular, $\trace{R}\in [0,n]$, as noted in \cite{MendoncaTorjeiro}. In a similar way, we can define
 \[\widetilde{L} = \dif\pi_1\circ f_{\ast} \in \Gamma(\TT^{\ast}M \oplus \TT\mathbb{Q}^{n_1}_{c_1,\mu_1}), \ \widetilde{K} = \dif\pi_1|_{\TT^{\perp}M}\in \Gamma((\TT^{\perp}M)^{\ast}\oplus \TT\mathbb{Q}^{n_1}_{c_1,\mu_1} ),\]
\[\widetilde{R} = \widetilde{L}^t \widetilde{L}, \ \widetilde{S} = \widetilde{K}^t \widetilde{L} \text{ \ and \ } \widetilde{T} = \widetilde{K}^t\widetilde{K},\]
where $\widetilde{R}$ and $\widetilde{T}$ are nonnegative symmetric operators whose eigenvalues lie in $[0,1]$, $\trace{\widetilde{R}}\in [0,n]$, and therefore
\[R+\widetilde{R}=\id_{\TT M}.\]
Under these notations, in \cite{LiraTojeiroVitorio10} the authors write the Gauss, Codazzi and Ricci equations as
\begin{multline}\label{Gauss-equation}
\mathcal{R}(X,Y)Z=\Big(c_1(X\wedge Y-X\wedge RY-RX\wedge Y)\\
+(c_1+c_2)RX\wedge RY\Big)Z+A_{\alpha(Y,Z)}X-A_{\alpha(X,Z)}Y,
\end{multline}
\begin{multline}\label{Codazzi-equation}
(\nabla_X^\perp \alpha)(Y,Z)-(\nabla_Y^\perp \alpha)(X,Z)=
c_1(\dotprod{X}{Z}SY-\dotprod{Y}{Z}SX)\\
+(c_1+c_2)(\dotprod{RY}{Z}SX-\dotprod{RX}{Z}SY),
\end{multline}
\begin{equation}\label{Ricci-equation}
\mathcal{R}^\perp(X,Y)\eta = \alpha(X,A_\eta Y)-\alpha(A_\eta X,Y)+(c_1+c_2)(SX\wedge SY)\eta,
\end{equation}
where $(X\wedge Y)Z= \dotprod{Y}{Z}X-\dotprod{X}{Z}Y$, for all $ X, Y,Z \in \mathfrak{X}(M)$.

In \cite{Dajczer-Gromoll-85}, Dajczer and Gromoll introduced the \emph{circular} isometric immersions of a Kähler manifold. Those isometric immersions are also known in the literature as \emph{$(1,1)$-geodesic} immersions and also as \emph{pluriharmonic} immersions  \cite{FerreiraRigoliTribuzy93,Udagawa}. In this work, we adopt the pluriharmonic terminology, and we recall that an isometric immersion $f:M^{2n}\to\Q^m$  of a Kähler manifold is said \emph{pluriharmonic} if the second fundamental form of $f$ satisfies
\[\alpha(X,JY) = \alpha(JX,Y), \text{ \ for all \ } X, Y \in \mathfrak{X}(M),\]
or equivalently, if the Weingarten operator $A_{\xi}$ of $f$ anticommutes with the almost complex structure $J$:
\[A_{\xi}J+JA_{\xi} = 0, \text{ \ for any \ }\xi\in \mathfrak{X}(M)^{\perp}.\] In particular, pluriharmonic immersions are minimal.

It is important to point out that for $n=1$, any orientable Riemannian surface $(\Sigma,\dif s^2)$ has a natural almost complex structure $J$. This structure is given by the rotation of angle $\pi/2$ on the tangent bundle $\TT\Sigma$ of $\Sigma$. Since for any $\xi\in\mathfrak{X}(\Sigma)^{\perp}$ we have $A_\xi J+JA_\xi = (\trace A_\xi) J$, then $f:\Sigma\to\Q^m$ is pluriharmonic if and only if $\trace{A_\xi}=0$ for any $\xi\in\mathfrak{X}(\Sigma)^{\perp}$, that is, if $f$ is a minimal immersion.

\section{Obstruction results}\label{obstruction-results}

In this section, we study the conditions to a minimal isometric immersion of a Kähler manifold into $\Q^m$ be a pluriharmonic immersion. As a consequence, we analyse the obstruction conditions to the existence of pluriharmonic immersions of Kähler manifolds into $\Q^m$.

In \cite{DajczerRodriguez86}, Dajczer and Rodríguez studied those kinds of problems for space forms $\Q^m_c$.  They concluded that pluriharmonic submanifolds and minimal submanifolds are the same objects in the Euclidean space. On the other hand, in the hyperbolic case, they showed that only surfaces can be immersed under assumption of minimality; and, in the spherical case, only surfaces can be immersed under assumption of pluriharmonicity. More generally in \cite{FerreiraRigoliTribuzy93}, Ferreira, Rigoli and Tribuzy show that pluriharmonic submanifolds are also equivalent to minimal submanifolds in locally symmetric Riemannian manifold of non-compact type.

For a given $f:M^{2n}\to\Q^m$ minimal isometric immersion of a Kähler manifold, our results are based in a \emph{pluriharmonicity property}: an equation that must be satisfied by the tensor $R$. Before finding this equation, we recall a general characterization of isometric immersions into slices of products of space forms proved by Mendonça and Tojeiro, \cite[Proposition 8]{MendoncaTorjeiro}. In terms of the trace of $R$, theses submanifolds are those on which either $\trace R = 0$ or $\trace R = \textrm{dim}M$. It is important to notice that, for their result, any assumption about almost complex structure on $M$ is required.

\begin{proposition}\label{prop-slice}
Let $f: M^n \to \Q^m$ be an isometric immersion. Then $f(M^n) \subset \Q^{n_1}_{c_1}\times\{p\}$ for some $p\in  \Q^{n_2}_{c_2}$ $($resp. $f(M^n) \subset \{p\}\times\Q^{n_2}_{c_2}$ for some $p\in  \Q^{n_1}_{c_1})$, if and only if $\trace{R} = 0$ $($resp. $\trace R=n)$.
\end{proposition}
\begin{proof}
Since $R+\widetilde{R}=\id_{\TT M}$, then $\trace R+\trace\widetilde{R}=n$.  Moreover, the eigenvalues of $R$ and $\widetilde{R}$ lies in $[0,1]$, $\|L\|^2 = \trace{R}$ and $\|\widetilde{L}\|^2 =  \trace{\widetilde{R}}$. Therefore, $\trace R=0$ if, and only if, $L=0$; and $\trace R=n$ if, and only if, $\trace \widetilde{R}=0$, i.e., $\widetilde{L}=0$. By definition of $L$, $\dif\pi_2\circ f_{\ast}=0$ holds if, and only if, $f(M^n) \subset \Q^{n_1}_{c_1}\times\{p\}$ for some $p\in \Q^{n_2}_{c_2}$; and by definition of $\widetilde{L}$,  $\dif\pi_1\circ f_{\ast}=0$ holds if, and only if, $f(M^n) \subset \{p\}\times\Q^{n_2}_{c_2}$ for some $p\in  \Q^{n_1}_{c_1}$.
\end{proof}

In the next result, we discuss a necessary and sufficient condition to a minimal isometric immersion of a  Kähler manifold into $\Q^m$ be a pluriharmonic immersion.

\begin{lemma}[Pluriharmonicity property]\label{lemma-pluri-prop} Let $f:M^{2n}\to\Q^m$ be a minimal isometric immersion of a
Kähler manifold. Then $f$ is pluriharmonic if and only if the tensor $R$ satisfies the following equation
\begin{equation}\label{eq-pluri-property}
4c_1(n-1)(n-\trace{R})+(c_1+c_2)\Big((\trace{R})^2-\|R\|^2-\dotprod{RJ}{JR}\Big)=0.
\end{equation}
\end{lemma}
\begin{proof}
At a point $p\in M$, we consider an orthonormal basis $\{X_1,\ldots,X_{2n}\}$ of $\mathrm{T}_p M$ such that $X_{2j}=JX_{2j-1}$, for $1\leq j\leq n$. Then, at this point, by the Gauss equation \eqref{Gauss-equation} we get
\begin{multline*}
\dotprod{\mathcal{R}(X_j,X)Y}{X_j} =\dotprod{\alpha(X,Y)}{\alpha(X_j,X_j)}-\dotprod{\alpha(X,X_j)}{\alpha(Y,X_j)}\\
+c_1\Big(\dotprod{X}{Y}-\dotprod{X}{X_j}\dotprod{Y}{X_j}
-\dotprod{X}{Y}\dotprod{RX_j}{X_j}\\
-\dotprod{RX}{Y}
+\dotprod{RX}{X_j}\dotprod{Y}{X_j}+\dotprod{Y}{RX_j}\dotprod{X}{X_j}\Big)\\
+(c_1+c_2)\Big(\dotprod{RX_j}{X_j}\dotprod{RX}{Y}-\dotprod{RX}{X_j}\dotprod{RX_j}{Y}\Big).
\end{multline*}
Since  $f$ is minimal, for $X=Y=X_i$, summing in $j$ from $1$ to $2n$, we have
\begin{multline}\label{exp-ricci-1}
\Ric(X_i) =
-\sum_{j=1}^{2n}\|\alpha(X_i,X_j)\|^2\\
+c_1\Big(2n-1-\trace{R}-2(n-1)\dotprod{RX_i}{X_i}\Big)\\
+(c_1+c_2)\Big(\dotprod{RX_i}{X_i}\trace{R}-\|RX_i\|^2\Big),
\end{multline}
and, similarly for $X=Y=JX_i$,
\begin{multline}\label{exp-ricci-2}
\Ric(JX_i) =
-\sum_{j=1}^{2n}\|\alpha(JX_i,X_j)\|^2\\
+c_1\Big(2n-1-\trace{R}-2(n-1)\dotprod{RJX_i}{JX_i}\Big)\\
+(c_1+c_2)\Big(\dotprod{RJX_i}{JX_i}\trace{R}-\|RJX_i\|^2\Big).
\end{multline}

On the other hand, the Kähler structure on $M$ implies that
\[\dotprod{\mathcal{R}(X_j,X_i)X_i}{X_j} = \dotprod{\mathcal{R}(X_j,X_i)JX_i}{JX_j},\]
and therefore, by the Gauss equation, summing in $j$ from $1$ to $2n$, we get
\begin{multline*}
\Ric(X_i) =
\sum_{j=1}^{2n} \dotprod{\alpha(X_i,JX_i)}{\alpha(X_j,JX_j)}
-\sum_{j=1}^{2n} \dotprod{\alpha(X_i,JX_j)}{\alpha(X_j,JX_i)}\\
+c_1\Big(1-\dotprod{RX_i}{X_i}-\dotprod{RJX_i}{JX_i}\Big)\\
+(c_1+c_2)\Big(\dotprod{RJX_i}{JRX_i} -  \dotprod{RX_i}{JX_i}\trace{JR}\Big).
\end{multline*}
However, we notice that the first term of expression above is equal to zero, because of $X_{2j}=JX_{2j-1}$, for $1\leq j\leq n$. Moreover, since $R$ is symmetric and $J$ anti-symmetric, we have $\trace{JR} =0$. Thus,
\begin{multline}\label{exp-ricci-3}
\Ric(X_i) =
-\sum_{j=1}^{2n} \dotprod{\alpha(X_i,JX_j)}{\alpha(X_j,JX_i)}\\
+c_1\Big(1-\dotprod{RX_i}{X_i}-\dotprod{RJX_i}{JX_i}\Big)
+(c_1+c_2)\dotprod{RJX_i}{JRX_i}.
\end{multline}

Consider $E =\bigoplus_{j=1}^{2n} \TT_p^{\perp}M$ endowed with the standard inner product. We set
\begin{align*}
  u_i &= (\alpha(X_i ,JX_1),\alpha(X_i ,JX_2),\ldots, \alpha(X_i ,JX_{2n})),\\
  v_i &= (\alpha(X_1,JX_i),\alpha(X_2,JX_i),\ldots, \alpha(X_{2n},JX_i)).
\end{align*}
Since $\Ric(JX,JY)=\Ric(X,Y)$, for all $X, Y \in \mathfrak{X}(M)$ and $|u_i-v_i|^2 = |u_i|^2+|v_i|^2-2\dotprod{u_i}{v_i}$, by equations \eqref{exp-ricci-1}, \eqref{exp-ricci-2} and \eqref{exp-ricci-3} we have
\begin{multline*}
|u_i-v_i|^2  = 2c_1\Big(2(n-1)-\trace{R}-(n-2)\big(\dotprod{RX_i}{X_i}+\dotprod{RJX_i}{JX_i}\big)\Big)\\
   +(c_1+c_2)\Big(\big(\dotprod{RX_i}{X_i}+\dotprod{RJX_i}{JX_i}\big)\trace{R}\\
   -\|RX_i\|^2-\|RJX_i\|^2-2\dotprod{RJX_i}{JRX_i}\Big),
\end{multline*}
for $1\leq i\leq n$. Then,
\begin{equation*}%\label{pluri-prop}
\frac{1}{2}\sum_{i=1}^{2n} |u_i-v_i|^2 = 4c_1(n-1)(n-\trace R)+(c_1+c_2)\Big((\trace R)^2-\|R\|^2-\dotprod{RJ}{JR}\Big),
\end{equation*}
that is, $|u_i-v_i|^2=0$ for all $1\leq i\leq n$ if, and only if, equation \eqref{eq-pluri-property} holds. Therefore, observing that $|u_i-v_i|^2=0$ for all $1\leq i\leq n$ if, and only if, $f$ is a pluriharmonic immersion, we conclude our assertion.
\end{proof}

\begin{remark} We notice that since $R+\widetilde{R}=\id_{\TT M}$, we obtain an analogous Pluriharmonicity property's Lemma for the tensor $\widetilde{R}$:
\begin{equation*}
4c_2(n-1)(n-\trace{\widetilde{R}})+(c_1+c_2)\Big((\trace{\widetilde{R}})^2-\|\widetilde{R}\|^2-\dotprod{\widetilde{R}J}{J\widetilde{R}}\Big)=0.
\end{equation*}
\end{remark}

\begin{remark}\label{lawn-remark} By the approach used in \cite{LawnRoth}, we show that the pluriharmonic property is given by
\[c_1\Big((\trace{R})^2-\|R\|^2-\dotprod{RJ}{JR}\Big)+c_2\Big((\trace{\widetilde{R}})^2-\|\widetilde{R}\|^2-\dotprod{\widetilde{R}J}{J\widetilde{R}}\Big)=0,\]
which is equivalent to the one provided by Lemma \ref{lemma-pluri-prop}, by the relation $R+\widetilde{R}=\id_{\TT M}$. Moreover, the Pluriharmonic property's Lemma can be generalised for minimal isometric immersions of a Kähler manifolds into multiproducts of space forms $\Q_{c_1}^{n_1}\times\cdots\times\Q_{c_k}^{n_k}$. In this case, the pluriharmonic property is given by
\[\sum_{j=1}^{k}c_j\Big((\trace{R_j})^2-\|R_j\|^2-\dotprod{R_jJ}{JR_j}\Big)=0,\]
where $R_j$ is the $f_j$ symmetric tensor that appear in  \cite{LawnRoth}. In particular, for the case $\Q^{n_1}_{c_1}\times\Q^{n_2}_{c_2}$, we have that $R_1=\widetilde{R}$ and $R_2=R$.
\end{remark}

As a consequence of Pluriharmonicity property's Lemma, we analyse the obstruction conditions to the existence of pluriharmonic immersions of Kähler manifolds into $\Q^{m-1}_c\times\R$. %Theses kinds of problems were also studied by Dajczer, Gromoll and Rodríguez in the cases of space forms \cite{DGGaussMap,DajczerRodriguez86}. In the spherical and hyperbolic case, only surfaces can be immersed under assumption of pluriharmonicity.

\begin{theorem}\label{thm-obstruction-QmxR}
Let $f:M^{2n}\to\Q^{m-1}_c\times\R$ be an isometric immersion of a Kähler manifold, with $c \neq 0$. Assume that
\begin{itemize}
\item[$\mathrm{i)}$] either $c < 0$ and $f$ is minimal;
\item[$\mathrm{ii)}$] or $c > 0$ and $f$ is pluriharmonic.
\end{itemize}
Then $n = 1$.
\end{theorem}
\begin{proof}
 Firstly, since $\mathbb{H}_c^{m-1}\times \R  $ is a locally symmetric Riemannian manifold of non-compact type follows from \cite[Proposition 1]{FerreiraRigoliTribuzy93} that $f$ is also a pluriharmonic immersion. Moreover, for $\Q^m = \Q^{m-1}_c\times\R$ we have that
$RX= \dotprod{X}{\partial_t^\top}\partial_t^\top$, where $\partial_t^\top$ is the projection of the unit vertical vector $\partial_t$ (corresponding to the factor $\R$) onto $\mathrm{T}M$. By Pluriharmoniciy property's Lemma, we have
\[4(n-1)\big(n-\|\partial_t^\top\|^2\big)=0,\]
that is, either $n=1$ or $\|\partial_t^\top\|^2=n$. Since $\|\partial_t^\top\|^2\leq 1$, in both cases we get $n=1$.
\end{proof}

\begin{remark}
We point out that Theorem \ref{thm-obstruction-QmxR} was also obtained by de Almeida in her thesis \cite[Theorem 3.1]{Kelly}, using similar methods. %in an unpublished result.
%é bom falar as diferenças
\end{remark}
\begin{remark}
We notice that Theorem \ref{thm-obstruction-QmxR} can be extended for an isometric immersion of a Kähler manifold $M^{2n}$ into a warped product manifold $I\times_{\rho}\mathbb{Q}^{m-1}_c$ endowed with the metric $\dif s^2 = \dif t^2 +\rho(t)^2 \dif\theta^2$, where $I\subset\R$ is an interval, $\rho: I\to \R$ is a non-constant positive smooth function and $\dif\theta^2$ denotes the metric of $\mathbb{Q}^{m-1}_c$. Indeed, by the Gauss equation (cf.  \cite{ChenXiang,ribeiro2019}), we compute the pluriharmonicity property by
\begin{equation}\label{casowarped}
4(n-1)\big(n\lambda(t)-\|\partial_t^\top\|^2\mu(t)\big)=0,
\end{equation}
where $\lambda(t) =\dfrac{c-\rho'(t)^2}{\rho(t)^2} \,\,\, \mbox{and}\,\,\, \mu(t)=\dfrac{c-\rho'(t)^2}{\rho(t)^2}+\dfrac{\rho''(t)}{\rho(t)}.$ However, we observe that when either $\rho''(t)\geq 0$ and $c\leq 0$, or $\rho''(t)\leq 0$ and $c>\rho'(t)^2$, then
\[n\lambda(t)-\|\partial_t^\top\|^2\mu(t)=\big(n-\|\partial_t^\top\|^2\big)\dfrac{c-\rho'(t)^2}{\rho(t)^2}-\|T\|^2\dfrac{\rho''(t)}{\rho(t)}= 0\]
if, and only if, either $\rho''(t)= 0$ and $\|\partial_t^\top\|^2=n$, that is, if $\{(t,\rho(t)) : t\in I\}\subset\R^2$ is a line and $n=1$, since $\|\partial_t^\top\|^2\leq 1$. Therefore, if $f:M^{2n}\to I\times_{\rho}\mathbb{Q}^{m-1}_c$ is a pluriharmonic immersion of a Kähler manifold such that either $\rho''(t)\geq 0$ and $c\leq 0$, or $\rho''(t)\leq 0$ and $c>\rho'(t)^2$, the pluriharmonicity property \eqref{casowarped} implies that $n=1$.
\end{remark}

\begin{corollary}\label{3.8}
Let $f:M^{2n}\to\Q^{n_1}_{c}\times\Q^{n_2}_{-c}$ be a pluriharmonic immersion of a Kähler manifold, with $c\neq 0$. Then either $\trace R = n$ or $n=1$.
\end{corollary}
\begin{proof}
Since $c_1+c_2 =0$, it follows directly of Pluriharmonicity property's Lemma that $4c(n-1)(n-\trace R)=0$.
\end{proof}

%\begin{corollary}Let $f:M^{2n}\to\Q^{m}_{c}\times\Q^{m}_{c}$ be a pluriharmonic isometric immersion of a Kähler manifold, with $c\neq 0$. If $\trace R \neq n$, then $n=1$.
%\end{corollary}
%\begin{proof}
%Firstly note that $R=\widetilde{R}$. Then, by Remark \ref{lawn-remark}, we have $(\trace{R})^2-\|R\|^2-\dotprod{RJ}{JR}=0$, and by Pluriharmonicity property's Lemma, we get our assertion.
%\end{proof}

\begin{corollary}\label{NonExistence2}
Let $f:M^{2n}\to\Q^{m}$ be a pluriharmonic immersion of a Kähler manifold. Assume that $RJ+JR =0$ $($resp. $RJ+JR =2J)$. Then either $c_1=0$ or $n=1$ $($resp. $c_2=0$ or $n=1)$.
\end{corollary}
\begin{proof} Since $RJ+JR =0$ if, and only if, $J^{-1}RJ+R =0,$ then $\trace{R} =0.$ Moreover, $\dotprod{RJ}{JR}=-\dotprod{RJ}{RJ}=-\|R\|^2$. By Pluriharmonicity property's Lemma, we have  $4c_1(n-1)n =0$, i.e., either $c_1 = 0$ or $n=1$.
Analogously, if $RJ+JR =2J$ then $\trace{R} =2n$ and $\trace{\widetilde{R}} =0$. Thus, by Pluriharmonicity property's Lemma for $\widetilde{R }$, we have  $4c_2(n-1)n =0$, i.e., either $c_2 = 0$ or $n=1$.
\end{proof}
\begin{remark} We observe that if $RJ+JR=0$ then $\trace{R}=0$, and by Proposition \ref{prop-slice}, we have $f(M^{2n})\subset\Q_{c_1}^{n_1}\times \{p\}$, for some $p\in  \Q^{n_2}_{c_2}$. By Corollary \ref{NonExistence2}, either $c_1=0$ and $f(M^{2n})\subset \R^{n_1}\times \{p\}$, or $n=1$ and $f(M^2)\subset\Q_{c_1}^{n_1}\times \{p\},$ that is, $f$ can be seen as isometric immersion into $\Q^{n_1}_{c_1}$ and then we recover a Dajczer-Rodríguez result presented in \cite{DajczerRodriguez86}.
\end{remark}

\section{Curvature estimates}\label{curvature-estimates}

The goal of this section  is to study upper bounds of the Ricci and scalar curvatures of Kähler manifolds, when we suppose the existence of minimal isometric immersions of this manifolds into some product of space forms.

In a classical work about minimal isometric immersions, Takahashi proved that the existence of minimal isometric immersion $f$ of an arbitrary Riemannian manifold $M^n$ into a space form $\Q^m_c$, $n\geq 2$, imposes an upper bound of the Ricci curvature of $M$, namely,
%space forms $\Q^m_c$, the Gauss equation provides naturally upper and lower bounds for given a minimal isometric immersion of an arbitrary Riemannian manifold $M^n$, $n\geq 2$, namely,
\begin{equation*}
\frac{n-1}{n}\big(cn-\|\alpha\|^2\big)\leq
\Ric \leq c(n-1),
\end{equation*}
%and therefore,
%\begin{equation*}
%  (n-1)\big(cn-\|\alpha\|^2\big)\leq \Scal \leq cn(n-1).
%\end{equation*}
%\begin{equation}\label{bound}
%\Ric\leq (n-1)c \text{ \ and \ } \Scal \leq n(n-1) c,
%\end{equation}
where $\|\alpha\|^2$ denotes the norm square of the second fundamental form \cite[Theorem 1]{Takahashi}. The equality case holds if and only if $f(M)$ is a totally geodesic submanifold of $\Q^m_c$. Dajczer and Rodríguez proved that the upper bound of the Ricci curvature of $M$ is more restrictive for a minimal isometric immersion $f$ of a Kähler manifold $M^{2n}$ into $\s^m_c$. In this case, they showed that $\Ric \leq cn$, with equality implying that $f$ has parallel second fundamental form \cite[Theorem 1.2]{DajczerRodriguez86}.

In order to improve the upper bounds of the Ricci and scalar curvatures of minimal Kähler submanifolds into some products space forms, we study firstly a general upper bound of the scalar curvature of those submanifolds in $\Q^{n_1}_{c_1}\times\Q^{n_2}_{c_2}$, given in terms of the tensor $R$. For this purpose, we said that an isometric immersion $f$ of a Kähler manifold $M^{2n}$ is an \emph{anti-pluriharmonic} immersion if the second fundamental form of $f$ satisfies
\[\alpha(X,JY) = -\alpha(JX,Y), \text{ \ for all \ } X, Y \in \mathfrak{X}(M),\]
or equivalently, if the Weingarten operator $A_{\xi}$ of $f$ commutes with the almost complex structure $J$:
\[A_{\xi}J=JA_{\xi}, \text{ \ for any \ }\xi\in \mathfrak{X}(M)^{\perp}.\]

Anti-pluriharmonic immersions into Euclidean space were firstly studied by Rettberg in \cite{Rettberg} and by Ferus in \cite{Ferus1980}, where was proved that anti-pluriharmonic immersions into $\Q^m_c$ have parallel second fundamental forms. We notice that, in an analogous way of holomorphic immersions, if the target space is also a Kähler manifold, then anti-holomorphic isometric immersions are anti-pluriharmonic immersions.
%
%This fact can be extended in the following sense:

%\begin{lemma}
%Let $f:M^{2n}\to\Q^m$ be an isometric immersion of a Kähler manifold with anti-pluriharmonic property. Then, the covariant derivative of $\alpha$ can be written in terms of the tensor $R$ and $S.$ That is, $$\Big(\nabla_X^{\perp} \alpha\Big)(Y,Z)=F(R,S)(X,Y,Z),$$ where
%\end{lemma}
%When $f$ lies in a slice we have that $R=0$ or $I$ and $S=0,$ and since
%$F(I,0) =F(0,0) =0$ follows that in any case $f$ has second fundamental form  parallel, recovering \cite{Ferus1980}.

%This approach will provides
%Indeed, we obtain a better upper bound for the Ricci curvature of minimal Kähler submanifolds in $\mathbb{S}_c^{m-1}\times\R$.

\begin{lemma}\label{lemma-ricci}
Let $f:M^{2n}\to\mathbb{Q}^m$ be a minimal isometric immersion of a Kähler manifold. Then the scalar curvature of $M$ satisfies
\[\Scal \leq 2nc_1(n-\trace{R})+\dfrac{(c_1+c_2)}{2}\Big((\trace{R})^2-\|R\|^2+\dotprod{RJ}{JR} \Big).\]
Moreover, the equality holds if, and only if, $f$ is an anti-pluriharmonic immersion.
%$f(M)\subset \mathbb{Q}_{c_1}^{n_1}\times\{t\}$ for some $t\in\mathbb{Q}_{c_2}^{n_2}$, and
%that \textcolor{red}{the second fundamental form is parallel.}
%In particular, if $c_1+c_2\geq 0,$ then
%\[\Scal \leq 2nc_1(n-\trace{R})+\dfrac{(c_1+c_2)}{2}(\trace{R})^2.\]
\end{lemma}
\begin{proof}
At a point $p\in M$, we consider an orthonormal basis $\{X_1,\cdots,X_{2n}\}$ of $\mathrm{T}_p M$ such that $X_{2j}=JX_{2j-1}$, for $1\leq j\leq n$.  Consider $E =\bigoplus_{j=1}^{2n} \TT_p^{\perp}M$ endowed with the standard inner product. We set
\begin{align*}
  u_i &= (\alpha(X_i ,JX_1),\alpha(X_i ,JX_2),\ldots, \alpha(X_i ,JX_{2n})),\\
  v_i &= (\alpha(X_1,JX_i),\alpha(X_2,JX_i),\ldots, \alpha(X_{2n},JX_i)).
\end{align*}
Since $\Ric(JX,JY)=\Ric(X,Y)$, for all $X, Y \in \mathfrak{X}(M)$,  by equations \eqref{exp-ricci-1}, \eqref{exp-ricci-2} and \eqref{exp-ricci-3}, and the Parallelogram identity, we get that
\begin{equation}\label{lemma-eq-ricci}
|u_i+v_i|^2 = -4\Ric(X_i)+2A_i+B_i+C_i,
\end{equation}
for $1\leq i\leq 2n$, where the coefficients $A_i$, $B_i$ and $C_i$ are given by
\begin{align*}
A_i  &= c_1\Big(1-\dotprod{RX_i}{X_i}-\dotprod{RJX_i}{JX_i}\Big)
        +(c_1+c_2)\dotprod{RJX_i}{JRX_i},\\
B_i  &= c_1\Big((2n-1)-\trace{R}-2(n-1)\dotprod{RX_i}{X_i}\Big)\\
        &\quad\quad\quad\quad\quad\quad\quad\quad\quad\quad\quad\quad\quad\quad
        +(c_1+c_2)\Big(\dotprod{RX_i}{X_i}\trace{R}-\|RX_i\|^2\Big),\\
C_i  &=c_1\Big((2n-1)-\trace{R}-2(n-1)\dotprod{RJX_i}{JX_i}\Big)\\
        &\quad\quad\quad\quad\quad\quad\quad\quad\quad\quad\quad\quad\quad\quad
        +(c_1+c_2)\Big(\dotprod{RJX_i}{JX_i}\trace{R}-\|RJX_i\|^2\Big).
\end{align*}
Thus, by equation \eqref{lemma-eq-ricci}, we obtain
%\begin{equation*}%\label{pluri-prop}
%\frac{1}{2}\sum_{i=1}^{2n} |u_i-v_i|^2 = 4c_1(n-1)(n-\trace R)+(c_1+c_2)\Big((\trace R)^2-\|R\|^2-\dotprod{RJ}{JR}\Big),
%\end{equation*}
\begin{equation}\label{ineq-ricci}
\Ric(X_i) \leq \frac{1}{4}(2A_i+B_i+C_i).
\end{equation}

On the other hand, we compute $2A_i+B_i+C_i$ by
\begin{multline*}
2A_i+B_i+C_i=2c_1\Big(2n-\trace{R}-n\big(\dotprod{RX_i}{X_i}+\dotprod{RJX_i}{JX_i}\big)\Big)\\
+(c_1+c_2)\Big(\big(\dotprod{RX_i}{X_i}+\dotprod{RJX_i}{JX_i}\big)\trace{R}\\
-\|RX_i\|^2-\|RJX_i\|^2+\dotprod{RJX_i}{JRX_i}\Big).
\end{multline*}
Therefore, summing in $i$ from $1$ to $2n$, we obtain
\begin{align*}
  4\Scal &\leq \sum_{i=1}^{2n}(2A_i+B_i+C_i) \\
   &= 8nc_1(n-\trace{R})+ (c_1+c_2)\Big(2(\trace{R})^2-2\|R\|^2-2\dotprod{RJ}{JR}\Big),
\end{align*}
that conclude our assertion.

Note that the equality holds if, and only if
\begin{equation*}
\sum_{i=1}^{2n} |u_i+v_i|^2 = 0,
\end{equation*}
that is, $u_i = -v_i$, for $1\leq j\leq n$, i.e.,
\begin{equation*}
\alpha(X,JY)+\alpha(JX,Y)=0
\end{equation*}
for $X,Y\in \mathfrak{X}(M)$, therefore, if, and only if, $f$ is anti-pluriharmonic.
%It is shown in \cite{Ferus1980} that this property implies that the second fundamental form $\alpha$ is parallel.
%If the equality holds on \eqref{ricci-inequality-3}, then by inequality \eqref{ricci-inequality-2}, we have that  $\dotprod{T}{X_i}^2+\dotprod{T}{JX_i}^2 =0,$ for $1\leq i \leq 2n$, i.e., $T=0$. Thus, $M^{2n}$ lies in a slice of $\mathbb{S}^{m-1}\times \{t\}$ for some $t\in \R$ and satisfies $\Ric = nc$. Therefore, by \cite[Theorem 1.2]{DajczerRodriguez86}, $M^{2n}$ has parallel second fundamental form.
\end{proof}
%In our next results, we show that the Ricci curvature and the Scalar curvature of a Kähler manifold $M^{2n}$ impose strong restrictions on the minimality of an isometric immersion of $M^{2n}$ into $\mathbb{S}_c^{m-1}\times\R$ and $\mathbb{S}_c^{m-k}\times \mathbb{H}_{-c}^k$.

In our next results, we show that the existence of a minimal isometric immersion of a Kähler manifold $M^{2n}$ into either $\mathbb{S}_c^{m-1}\times\R$ or $\mathbb{S}_c^{m-k}\times \mathbb{H}_{-c}^k$ imposes strong restrictions on the Ricci curvature and the scalar curvature of $M^{2n}$.

\begin{theorem}\label{thm4.1}
Let $f:M^{2n}\to\mathbb{S}_c^{m-1}\times\R$ be a minimal isometric immersion of a
Kähler manifold. Then the Ricci curvature of $M$ satisfies $\Ric \leq  c(2n-\|\partial_t^\top\|^2)/2$, with equality implying that $f(M^{2n})\subset \mathbb{S}_c^{m-1}\times\{t\}$ for some $t\in\R$, and that $f$ has parallel second fundamental form.
\end{theorem}
\begin{proof}
%At a point $p\in M$, we consider an orthonormal basis $\{X_1,\cdots,X_{2n}\}$ of $\mathrm{T}_p M$ such that $X_{2j}=JX_{2j-1}$, for $1\leq j\leq n$. Moreover,
In the case of $\mathbb{S}_c^{m-1}\times\R$, we have that
$RX= \dotprod{X}{\partial_t^\top}\partial_t^\top$, where $\partial_t^\top$ is the projection of the unit vertical vector $\partial_t$ (corresponding to the factor $\R$) onto $\mathrm{T}M$. Then, the coefficients $A_i$, $B_i$ and $C_i$ are given by
\begin{align*}
A_i &=c\Big(1-\dotprod{\partial_t^\top}{X_i}^2-\dotprod{\partial_t^\top}{JX_i}^2\Big),\\
B_i &= c\Big(2n-1-\|\partial_t^\top\|^2-2(n-1)\dotprod{\partial_t^\top}{X_i}^2\Big),\\
C_i &= c\Big(2n-1-\|\partial_t^\top\|^2-2(n-1)\dotprod{\partial_t^\top}{JX_i}^2\Big),
\end{align*}
for $1\leq i\leq 2n$. We compute $2A_i+B_i+C_i$ by
\begin{align*}
2A_i+B_i+C_i&=2c\Big(2n-\|\partial_t^\top\|^2-n(\dotprod{\partial_t^\top}{X_i}^2+\dotprod{\partial_t^\top}{JX_i}^2)\Big).
\end{align*}
Thus, by equation \eqref{ineq-ricci}, we obtain
\begin{equation}\label{ricci-inequality-2}
\Ric(X_i) \leq  \frac{c}{2}\Big(2n-\|\partial_t^\top\|^2-n(\dotprod{\partial_t^\top}{X_i}^2+\dotprod{\partial_t^\top}{JX_i}^2)\Big),
\end{equation}
and therefore,
\begin{equation}\label{ricci-inequality-3}
\Ric(X_i) \leq  \frac{c}{2}(2n-\|\partial_t^\top\|^2),
\end{equation}
for $1\leq i \leq 2n$ and $n\geq 1$.

If the equality holds on \eqref{ricci-inequality-3}, then by inequality \eqref{ricci-inequality-2}, we have that  $\dotprod{\partial_t^\top}{X_i}^2+\dotprod{\partial_t^\top}{JX_i}^2 =0,$ for $1\leq i \leq 2n$, i.e., $\partial_t^\top=0$. Thus, $f(M^{2n})$ lies into a slice $\mathbb{S}^{m-1}\times \{t\}$, for some $t\in \R$, and satisfies $\Ric = nc$. Therefore, by \cite[Theorem 1.2]{DajczerRodriguez86}, $f$ has parallel second fundamental form.
%By equations \eqref{exp-ricci-1}, \eqref{exp-ricci-2} and \eqref{exp-ricci-3}, we have
%\begin{equation}\label{exp-ricci-11}
%\Ric(X_i) =
%-\sum_{j=1}^{2n}\|\alpha(X_i,X_j)\|^2
%+c\Big((2n-1)-\|T\|^2-2(n-1)\dotprod{T}{X_i}^2\Big),
%\end{equation}
%\begin{equation}\label{exp-ricci-12}
%\Ric(JX_i) =
%-\sum_{j=1}^{2n}\|\alpha(JX_i,X_j)\|^2
%+c\Big((2n-1)-\|T\|^2-2(n-1)\dotprod{T}{JX_i}^2\Big)
%\end{equation}
%and
%\begin{equation}\label{exp-ricci-13}
%\Ric(X_i) =-\sum_{j=1}^{2n} \dotprod{\alpha(X_i,JX_j)}{\alpha(X_j,JX_i)}
%+c\Big(1-\dotprod{T}{X_i}^2-\dotprod{T}{JX_i}^2\Big).
%\end{equation}
%Consider $E =\bigoplus_{j=1}^{2n} \TT_p^{\perp}M$ endowed with the standard inner product. We set
%\begin{align*}
 % u_i &= (\alpha(X_i ,JX_1),\alpha(X_i ,JX_2),\ldots, \alpha(X_i ,JX_{2n}),\\
  %v_i &= (\alpha(X_1,JX_i),\alpha(X_2,JX_i),\ldots, \alpha(X_{2n},JX_i)).
%\end{align*}
%Then by equations \eqref{exp-ricci-11}, \eqref{exp-ricci-12}, \eqref{exp-ricci-13} and since $\Ric(JX,JY)=\Ric(X,Y)$, for all $X, Y \in \mathfrak{X}(M)$, and by parallelogram identity we have that,
%\begin{equation}\label{ricci-inequality}
%\|u_i+v_i\|^2 = -4\Ric(X_i)+2A+B+C,
%\end{equation}
%for $1\leq i\leq 2n$, where
%\begin{equation*}
%\Ric(X_i) \leq \frac{1}{4}(2A+B+C),
%\end{equation*}
%that is,
\end{proof}

\begin{remark} Given a minimal isometric immersion $f:M^n\to \Q^{m-1}_c\times\R$ of an arbitrary manifold $M^n$, $n\geq 2$, the Gauss equation provides a natural bound for the Ricci curvature, controlled by $\|\partial_t^\top\|^2$, on which $\|\partial_t^\top\|^2\leq 1$, in the following sense:
\begin{align*}
  \Ric&\leq c\big(n-1-\|\partial_t^\top\|^2\big), \text{ \ for $c>0$,} \\
  \Ric&\leq c(n-1)\big(1-\|\partial_t^\top\|^2\big), \text{ \ for $c<0$.}
\end{align*}
In both cases, the equality case holds if and only if either $f(M^2)$ is a totally geodesic surface in $\Q^{m-1}_c\times\R$, or $f(M^n)$ is a totally geodesic submanifold that lies into a slice of $\Q^{m-1}_c\times\R$.

When we assume that $M^{2n}$ is a Kähler manifold, with $n>1$, and $c>0$, this upper bound is less restrictive than the one provides by Theorem \ref{thm4.1}. However for $n=1$, the upper bound provided by the Gauss equation is more restrictive than the one provides our Theorem \ref{thm4.1}.
\end{remark}

\begin{remark}
We point out that the upper bound provides by Theorem \ref{thm4.1} also holds in $\mathbb{Q}_c^{m-1}\times\R$, with $c\in \R$. Moreover, when $c=0$, this upper bound is the same obtained by Gauss equation. For $c<0$, by the previous remark, we can check that this upper bound is less restrictive than the one provides by Gauss equation; and we recall that surfaces are the only minimal Kähler submanifolds in $\mathbb{H}_c^{m-1}\times\R$ (Theorem \ref{thm-obstruction-QmxR}).
\end{remark}

%The following corollary is straightforward computation.
%\begin{corollary}
%Let $f:M_{\widetilde{c}}^{2n}\to\mathbb{S}_c^{m-1}\times\R$ be a minimal isometric immersion of a
%Kähler manifold with constant sectional curvature $\widetilde{c}$, for $n\geq 1$. Then $\widetilde{c} \leq  \dfrac{c}{4n}(2n-\|T\|^2).$
%\end{corollary}

\begin{corollary}\label{thm4.2}
Let $f:M^{2n}\to\mathbb{S}_c^{m-1}\times\R$ be a minimal isometric immersion of a Kähler manifold. Then the scalar curvature of $M$ satisfies $\Scal\leq 2nc(n-\|\partial_t^\top\|^2)$. The equality holds if, and only if, $f$ is an anti-pluriharmonic immersion.
\end{corollary}
\begin{proof}
Since $RX= \dotprod{X}{\partial_t^\top}\partial_t^\top,$ by a direct computation we get that $\dotprod{RJ}{JR} =0$ and $\|R\|^2 = \|\partial_t^\top\|^4 =( \trace{R})^2$. Therefore, by Lemma \ref{lemma-ricci}, we obtain our assertion.
\end{proof}
%\begin{proof}
%Summing over $i$ the equation \eqref{ricci-inequality}, we obtain,
%\begin{equation*}
%\sum_{i=1}^{2n} |u_i+v_i|^2 = -4\Scal +8nc(n-\|T\|^2).
%\end{equation*}
%Therefore, $\Scal \leq 2nc(n-\|T\|^2).$
%Suppose that $M$ satisfies $\Scal = 2nc(n-\|T\|^2)$. Then,
%\begin{equation*}
%\sum_{i=1}^{2n} |u_i+v_i|^2 = 0,
%\end{equation*}
%that is, $u_i = -v_i$, for $1\leq j\leq n$. i.e.,
%\begin{equation*}
%\alpha(X,JY)+\alpha(JX,Y)=0
%\end{equation*}
%for $X,Y\in \mathfrak{X}(M)$. It is shown in \cite{Ferus1980} that this property implies that the second fundamental form $\alpha$ is parallel.
%\end{proof}

\begin{remark}
Theorem \ref{thm4.1} give us an upper bound of the scalar curvature of $M^{2n}$, precisely $\Scal \leq nc(2n-\|\partial_t^\top\|^2)$. However, this upper bound is less restrictive than the one provides by Corollary \ref{thm4.2}.
\end{remark}
%The theorem \ref{thm4.1} can be extended when the ambient space is $\mathbb{S}^{m-k}\times \mathbb{H}^k$ with additional hypothesis on the trace of $R.$ First, it is interesting to note that only pluriharmonic surfaces exist in $\mathbb{S}^{m-k}\times \mathbb{H}^k$ if trace of $R$ is different to $n,$ by Corollary \ref{3.8}.
%%
%On the other hand, it is possible construct examples of minimal immersions of $M^{2n}$ on $\mathbb{S}^{m-k}\times \mathbb{H}^k.$ It is enough to take product of two surface in each factor.

\begin{corollary}\label{ricci-SxQ}
Let $f:M^{2n}\to \mathbb{S}_c^{m-k}\times \mathbb{H}_{-c}^k$ be a minimal isometric immersion of a Kähler manifold. Then $\Ric \leq c(2n-\trace{R})/2,$ with equality implying $f(M)\subset \mathbb{S}_c^{m-k}\times\{p\}$ for some $p\in\mathbb{H}_{-c}^k$, $\Ric = cn$ and that $f$ has parallel second fundamental form.
\end{corollary}
\begin{proof}
Since $c_1=-c_2=c$, then the coefficients $A_i$, $B_i$ and $C_i$ are given by
\begin{align*}
A_i &= c\Big(1-\dotprod{RX_i}{X_i}-\dotprod{RJX_i}{JX_i}\Big),\\
B_i &= c\Big((2n-1)-\trace{R}-2(n-1)\dotprod{RX_i}{X_i}\Big),\\
C_i &= c\Big((2n-1)-\trace{R}-2(n-1)\dotprod{RJX_i}{JX_i}\Big),
\end{align*}
for $1\leq i\leq 2n$. We compute $2A_i+B_i+C_i$ by
\begin{align*}
2A_i+B_i+C_i&=2c\Big(2n-\trace{R}-n(\dotprod{RX_i}{X_i}+\dotprod{RJX_i}{JX_i})\Big).
\end{align*}
Thus, by equation \eqref{ineq-ricci}, we obtain
\begin{equation*}
\Ric(X_i) \leq  2c\Big(2n-\trace{R}-n(\dotprod{RX_i}{X_i}+\dotprod{RJX_i}{JX_i})\Big),
\end{equation*}
and therefore,
\begin{equation*}
\Ric(X_i) \leq  \frac{c}{2}(2n-\trace{R}),
\end{equation*}
for $1\leq i \leq 2n$, since $R$ is a non-negative operator.

If the equality holds, we have that  $\dotprod{RX_i}{X_i}+\dotprod{RJX_i}{JX_i} =0,$ for $1\leq i \leq 2n$, i.e., $\trace{R}=0$, and thus $M$ satisfies $\Ric = nc$. By Proposition \ref{prop-slice}, $f(M^{2n})$ lies into a slice $\mathbb{S}^{m-k}\times \{p\}$ for some $p\in \mathbb{H}^{k}$ and, therefore, by \cite[Theorem 1.2]{DajczerRodriguez86}, $f$ has parallel second fundamental form.
\end{proof}
% Incluir observação sobre o traço de R = \|pi_2 \circ f_*\|^2.
\begin{remark} Given a minimal isometric immersion $f:M^n\to \mathbb{Q}_c^{m-k}\times \mathbb{Q}_{-c}^k$ of an arbitrary manifold $M^n$, $n\geq 2$ and $c\neq 0$, the Gauss equation provides a natural bound for the Ricci curvature, controlled by $\trace R$, on which $0\leq \trace R\leq n$, in the following sense:
\begin{align*}
  \Ric&\leq c\big(n-1-\trace R\big), \text{ \ for $c>0$,} \\
  \Ric&\leq c(n-1)\big(1-\trace R\big), \text{ \ for $c<0$ \ and \ $\trace R \leq 1$,}\\
  \Ric&\leq c\big(1-\trace R\big), \text{ \ for $c<0$ \ and \ $\trace R>1$.}
\end{align*}
For either $c>0$ or $c<0$ and $\trace R < 1$, the equality case holds if and only if either $f(M^2)$ is a totally geodesic surface in $\mathbb{Q}_c^{m-k}\times \mathbb{Q}_{-c}^k$, or $f(M^n)$ is a totally geodesic submanifold that lies into a slice $\mathbb{Q}_c^{m-k}\times\{p\}$, for some $p\in\mathbb{Q}_{-c}^k$. For $c<0$ and $\trace R > 1$,  the equality case holds if and only if either $f(M^2)$ is a totally geodesic surface in $\h^{m-1}_c\times\mathbb{S}_{-c}^k$, or $f(M^n)$ is a totally geodesic submanifold that lies into a slice $\{p\}\times\mathbb{S}_{-c}^k$, for some $p\in\mathbb{H}_c^{m-k}$. Finally, for $c<0$ and $\trace R = 1$, the equality case holds if and only if $f(M^2)$ is a totally geodesic surface in $\mathbb{H}_c^{m-k}\times \mathbb{S}_{-c}^k$.

When we assume that $M^{2n}$ is a Kähler manifold, with $n>1$ and $c>0$, this upper bound is less restrictive than the one provides by Theorem \ref{thm4.1}. However for $n=1$, the upper bound provided by the Gauss equation is more restrictive than the one provides our Theorem \ref{thm4.1}.

\begin{remark}
We point out that the upper bound provides by Corollary \ref{ricci-SxQ} also holds in $\mathbb{H}_c^{m-k}\times \mathbb{S}_{-c}^k$. However, by the previous remark, we can check that this upper bound is less restrictive than the one provides by Gauss equation.
\end{remark}
\end{remark}

\begin{corollary}
Let $f:M^{2n}\to \mathbb{Q}_c^{m-k}\times \mathbb{Q}_{-c}^k$ be a minimal isometric immersion of a Kähler manifold, with $c\neq 0$. Then the scalar curvature of $M$ satisfies $\Scal \leq 2nc(n-\trace{R}).$ The equality holds if, and only if, $f$ is an anti-pluriharmonic immersion.
\end{corollary}

%\textcolor{red}{$\bullet$ Na igualdade vale que $f(M)$ mora em um slice?}

\begin{remark} In the previous section, we see that minimal isometric immersions into $\mathbb{Q}_c^{m-k}\times \mathbb{Q}_{-c}^k$ with $\trace{R}=n$ are not necessarily surfaces (Pluriharmoniciy property’s Lemma and Corollary \ref{3.8}). However, the latest corollaries give us some information about what occurs in this case; we obtain that $\Ric\leq nc/2$ and $\Scal\leq 0$.
\end{remark}

%\begin{remark} In the previous section, we see that minimal isometric immersions into $\mathbb{Q}_c^{m}\times \mathbb{Q}_{c}^m$ with $\trace{R}=n$ are not necessarily surfaces (Pluriharmoniciy property’s Lemma and Corollary \ref{3.8}). However, since $(\trace{R})^2-\|R\|^2-\dotprod{RJ}{JR}=0$ and $2 R= \id_{\TT M}$, by Lemma \ref{lemma-ricci}, we get $\Scal \leq cn(4n-1)/2$.
%\end{remark}

%\begin{corollary}
%Let $f:M^{2n}\to \mathbb{S}_c^{m-k}\times \mathbb{H}_{-c}^k$ be a minimal isometric immersion of a Kähler manifold, for $n\geq 2$ and $c>0$. Then $\Scal\leq 0,$ with equality implying $f(M)\subset \mathbb{S}_c^{m-k}\times\{t\}$ for some $t\in\mathbb{H}_{-c}^k$, and that the second fundamental form is parallel.
%\end{corollary}
\begin{remark} Let $f:M^{2n}\to \widetilde{M}^m$ be a minimal isometric immersion of a Kähler manifold into an arbitrary Riemannian manifold $\widetilde{M}^m$ and denote by $\widetilde{R}$ the Riemann curvature tensor of $\widetilde{M}$. In the general case, with the conventions used in the proofs of Lemma \ref{lemma-pluri-prop} and Lemma \ref{lemma-ricci}, our results are obtained by the study of the quantities $\omega_{i,-}=|u_i-v_i|^2$ and $\omega_{i,+}=|u_i+v_i|^2$, given by
\begin{multline*}
\omega_{i,\pm} = -2\big(\Ric(X_i)\pm\Ric(X_i)\big)
+\sum_{j=1}^{2n}\Big(\dotprod{\widetilde{R}(X_j,X_i)X_i}{X_j}\\
\pm2\dotprod{\widetilde{R}(X_j,X_i)JX_i}{JX_j}
+\dotprod{\widetilde{R}(X_j,JX_i)JX_i}{X_j}\Big),
\end{multline*}
where $\{X_1,\cdots,X_{2n}\}$ is an orthonormal basis of $\mathrm{T}_p M$, such that $X_{2j}=JX_{2j-1}$, for $1\leq j\leq n$. Then $\sum_{i=1}^{2n}\omega_{i,-}=0$ is the pluriharmonicity property and $\omega_{i,+}\geq 0$ provides the Ricci and scalar estimates for $M^{2n}$. When $\widetilde{M}^m$ is a conformally flat Riemannian manifold, its Riemann curvature tensor is given by
\begin{multline*}\dotprod{\widetilde{R}(X, Y)Z}{W} = \mathcal{S}(X,W)\dotprod{Y}{Z}+\mathcal{S}(Y,Z)\dotprod{X}{W}\\
-\mathcal{S}(X,Z)\dotprod{Y}{W}- \mathcal{S}(Y,W)\dotprod{X}{Z},
\end{multline*}
where $\mathcal{S}$ is the \emph{Schouten tensor} of $\widetilde{M}^m$, defined by
\begin{equation*}
\mathcal{S}(X,Y) = \frac{1}{m-2}\Bigg(\widetilde{\Ric} (X,Y) - \frac{\widetilde{\Scal}}{2(m-1)}\dotprod{X}{Y}\Bigg),
\end{equation*}
for $X, Y,Z,W \in \mathfrak{X}(\widetilde{M}).$ If $\mathcal{S}|_{\TT M}$ denotes the restriction of $\mathcal{S}$ to $\TT M\times \TT M$, then
\begin{align*}
\sum_{i=1}^{2n} \omega_{i,-} &= 8(n-1)\trace{\mathcal{S}|_{\TT M}},\\
\sum_{i=1}^{2n} \omega_{i,+} &=4\Big(2n\trace{\mathcal{S}|_{\TT M}}-\Scal\Big).
\end{align*}
Therefore, if $f:M^{2n}\to \widetilde{M}^m$ is a minimal isometric immersion of a Kähler manifold into a conformally flat Riemannian manifold $\widetilde{M}^m$ then $\Scal \leq 2n\trace{\mathcal{S}|_{\TT M}},$ where the equality holds if, and only if, $f$ is an anti-pluriharmonic immersion. Moreover, if $f$ satisfies $\trace{\mathcal{S}|_{\TT M}}\neq 0$ then it is pluriharmonic if, and only if, $n=1$. We observe that special cases satisfying this trace assumption were studied by Ferreira, Rigoli and Tribuzy in \cite{FerreiraRigoliTribuzy93}.
\end{remark}

\bibliographystyle{amsplain}
\bibliography{references}

\end{document}